\documentclass[11pt]{article}
\usepackage{amsmath, amsfonts, amssymb, amsthm,  graphicx, mathtools, enumerate, dsfont}

\usepackage[blocks, affil-it]{authblk}

\allowdisplaybreaks

\usepackage[square]{natbib}
\usepackage[CJKbookmarks=true,
            bookmarksnumbered=true,
			bookmarksopen=true,
			colorlinks=true,
			citecolor=red,
			linkcolor=blue,
			anchorcolor=red,
			urlcolor=blue]{hyperref}
\usepackage[usenames]{color}

\usepackage[letterpaper, left=1.2truein, right=1.2truein, top = 1.2truein, bottom = 1.2truein]{geometry}
\usepackage[ruled, vlined, lined, commentsnumbered]{algorithm2e}

\usepackage{booktabs}

\usepackage{prettyref,soul}

\usepackage{hyperref}
    \hypersetup{ colorlinks = true, linkcolor = blue }

\newtheorem{thm}{Theorem}[section]
\newtheorem{example}{Example}[section]
\newtheorem{definition}{Definition}[section]

\newrefformat{eq}{(\ref{#1})}
\newrefformat{chap}{Chapter~\ref{#1}}
\newrefformat{sec}{Section~\ref{#1}}
\newrefformat{algo}{Algorithm~\ref{#1}}
\newrefformat{fig}{Fig.~\ref{#1}}
\newrefformat{tab}{Table~\ref{#1}}
\newrefformat{rmk}{Remark~\ref{#1}}
\newrefformat{clm}{Claim~\ref{#1}}
\newrefformat{def}{Definition~\ref{#1}}
\newrefformat{cor}{Corollary~\ref{#1}}
\newrefformat{lmm}{Lemma~\ref{#1}}
\newrefformat{lemma}{Lemma~\ref{#1}}
\newrefformat{prop}{Proposition~\ref{#1}}
\newrefformat{app}{Appendix~\ref{#1}}
\newrefformat{ex}{Example~\ref{#1}}
\newrefformat{exer}{Exercise~\ref{#1}}
\newrefformat{soln}{Solution~\ref{#1}}
\newrefformat{cond}{Condition~\ref{#1}}

\def\P{{\mathbb P}}
\def\ind{\mathds{1}}
\def\Cov{{\rm Cov}} 

\def\text#1{\mbox{\rm #1}}

\newcommand{\norm}[1]{\|{#1} \|}

\newcommand{\supp}{{\rm supp}}

\newcommand{\cF}{\mathcal{F}}
\newcommand{\reals}{\mathbb{R}}
\newcommand{\E}{\mathbb{E}}
\newcommand{\Z}{\mathbb{Z}}
\newcommand{\cM}{\mathcal{M}}
\newcommand{\cL}{\mathcal{L}}

\newcommand{\cA}{\mathcal{A}}
\newcommand{\cB}{\mathcal{B}}
\newcommand{\cX}{\mathcal{X}}
\newcommand{\R}{\reals}
\newcommand{\card}{\rm Card}

\def\Lip{\mathrm{Lip}}
\def\N{{\mathbb N}}

\let\hat\widehat

\let\tilde\widetilde

\begin{document}

\setlength{\abovedisplayskip}{5pt}
\setlength{\belowdisplayskip}{5pt}
\setlength{\abovedisplayshortskip}{5pt}
\setlength{\belowdisplayshortskip}{5pt}

\title{\LARGE Probability inequalities for high dimensional time series under a triangular array framework}

\author[1]{Fang Han}
\author[2]{Wei Biao Wu}
\affil[1]{
University of Washington
 \\
fanghan@uw.edu
}
\affil[2]{
University of Chicago
 \\
wbwu@galton.uchicago.edu
}

\date{}

\maketitle

\begin{abstract}
Study of time series data often involves measuring the strength of temporal dependence, on which statistical properties like consistency and central limit theorem are built. Historically, various dependence measures have been proposed. In this note, we first survey some of the most well-used dependence measures as well as various probability and moment inequalities built upon them under a high-dimensional triangular array time series setting. We then argue that this triangular array setting will pose substantially new challenges to the verification of some dependence conditions. In particular, ``textbook results" could now be misleading, and hence are recommended to be used with caution. 
\end{abstract}


\section{The measure of dependence}\label{sec:intro}

We first introduce the mixing conditions defined on $\sigma$-fields. Fix the probability space as $(\Omega, \cF, \P)$. For any two $\sigma$-fields $\cA, \cB$ belonging to $\cF$, define the following four measures of dependence between $\cA$ and $\cB$ (cf. Chapter 3, \cite{bradley2007}):
\begin{align*}
\alpha(\cA,\cB):=\sup_{A\in\cA, B\in\cB}|\P(A\cap B)-\P(A)\P(B)|,\\
\beta(\cA,\cB):=\sup \frac{1}{2}\sum_{i=1}^I\sum_{j=1}^J|\P(A_i\cap B_j)-\P(A_i)\P(B_j)|,\\
\phi(\cA,\cB):=\sup_{A\in\cA,B\in\cB, \P(A)>0}|\P(B~|~A)-\P(B)|,\\
\rho(\cA,\cB):=\sup\{|\mathrm{Corr}(f,g)|,~~f\in \cL^2_{\reals}(\cA), g\in \cL^2_{\reals}(\cB)\},
\end{align*}
where the supremum in the definition of $\beta(\cA,\cB)$ is taken over all pairs of partitions $\{A_1,\ldots,A_I\}$ and $\{B_1,\ldots,B_J\}$ such that $A_i\in\cA$ and $B_j\in\cB$ for all $i,j$, and for any $p\in[1,\infty]$, let $\cL^p_{\reals}(\cA)$ represent the family of all real valued, $\cA$-measurable random variables $X$ on $\Omega$ such that $\norm{X}_{L_p}:=(\E |X|^p)^{1/p}<\infty$. We refer to \cite{bradley2005basic} for basic properties and historical developments on these dependence measures.

Now let's consider a (not necessarily stationary) time series $\{X_t\in \R^d\}_{t\in\Z}$ with $\Z$ and $\R^d$ representing the sets of all integers and all $d$-dimensional real vectors. For each ``time gap" $m=1,2,\ldots$, with the above dependence measures, we are now ready to define the following four mixing coefficients that appear frequently in literature:
\begin{align*}
\alpha(\{X_t\}_{t\in\Z};m) := \sup_{j\in\Z}\alpha(\sigma(\{X_t\}_{t\leq j}), \sigma(\{X_t\}_{t\geq j+m})),\\
\beta(\{X_t\}_{t\in\Z};m) := \sup_{j\in\Z}\beta(\sigma(\{X_t\}_{t\leq j}), \sigma(\{X_t\}_{t\geq j+m})),\\
\phi(\{X_t\}_{t\in\Z};m) := \sup_{j\in\Z}\phi(\sigma(\{X_t\}_{t\leq j}), \sigma(\{X_t\}_{t\geq j+m})),\\
\rho(\{X_t\}_{t\in\Z};m) := \sup_{j\in\Z}\rho(\sigma(\{X_t\}_{t\leq j}), \sigma(\{X_t\}_{t\geq j+m})).
\end{align*}
Here for any random variable $X$, $\sigma(X)$ is understood to be the $\sigma$-field generated by $X$. A review of the history of these mixing coefficients can be found in Section 2.1 in \cite{bradley2005basic}. We also refer readers to the books of \cite{doukhan1994mixing}, \cite{bradley2007}, and \cite{rio2017asymptotic}.

The above mixing coefficients are defined on $\sigma$-fields, and are usually difficult to be explicitly calculated in practice (though when the model is fixed, asymptotic bounds on coefficients can be derived for many time series models and have been established in many works). This is part of the reason to define weak dependence measures that are often much easier to calculate. In the following we introduce several of the most well-used ones. 

\cite{bickel1999new} and \cite{doukhan1999new} introduced a notion of weak dependence that facilitates explicit calculation of the independence strength between ``past" and ``future" without resorting to the latent $\sigma$-fields. They could be roughly understood as upper bounding
\[
{\rm Cov}(f({\rm `` past"}), g({\rm ``future"}))
\] 
by the gap between ``past" and ``future" as well as some parameters of the functions $f$ and $g$. In detail, 
for a function $g:(\R^d)^u \rightarrow \R$, let's define 
\begin{align*}
\Lip_\delta g \!:=\! \sup\Bigl\{ \frac{|g(x_1, \ldots, x_u) \!-\! g(y_1, \ldots, y_u)|}{\delta((x_1, \ldots, x_u), (y_1, \ldots, y_u))}: (x_1, \ldots, x_u) \!\neq\! (y_1, \ldots, y_u) \Bigr\}, 
\end{align*}
where $\delta(\cdot)$ represents a certain metric on the real vector space. Denote $\Lambda_\delta := \{g \!:\! (\R^u)^d \rightarrow \R \text{ for some $u$ }: \Lip_\delta g < \infty\}$ and $\Lambda^{(1)}_\delta := \{g \in \Lambda_\delta: \|g\|_{\infty} \leq 1\}$ with $\|g\|_{\infty} := \sup_{x} |g(x)|$. In the following, $\N$ represents the set of all natural numbers.

\begin{definition}[\cite{doukhan1999new,Doukhan2007}]\label{def:weakdependence}
The process $\{X_t\}_{t \in \Z}$ is $(\Lambda_\delta^{(1)},\psi,\zeta)$-weakly dependent if and only if there exists a function $\psi: \R^2_{+} \times \N^2 \rightarrow \R_{+}$ and a sequence $\zeta=\{\zeta(n)\}_{n \geq 0}$ decreasing to 0 as $n$ goes to infinity, such that for any $g_1,g_2 \in \Lambda_\delta^{(1)}$ with $g_1: (\R^d)^u \rightarrow \reals$, $g_2: (\R^d)^v \rightarrow \R$, $u,v \in \N$, and any $u$-tuple $(s_1,\ldots,s_u)$ and any $v$-tuple $(t_1,\ldots,t_v)$ with $s_1\leq \cdots \leq s_u < t_1 \leq \cdots \leq t_v$, the following inequality is satisfied:
\begin{align*}
\Bigl| \Cov\Bigl\{ g_1(x_{s_1}, \ldots, x_{s_u}), g_2(x_{t_1}, \ldots, x_{t_v}) \Bigr\} \Bigr| \leq \psi( \Lip_\delta g_1, \Lip_\delta g_2, u,v ) \zeta(t_1 - s_u).
\end{align*}
\end{definition}
Important examples of $(\Lambda_\delta^{(1)},\psi,\zeta)$-weakly dependent processes include $\theta$-, $\eta$-, $\kappa$-, and $\lambda$-dependences, which are listed in Table \ref{tab:WDeg}. They correspond to different choices of the function $\psi$. Similar to the mixing coefficients, the sequence $\zeta$ describes the degree of dependence over the process. 

\begin{table}[t]
\caption{Important examples of weak dependence.}
\label{tab:WDeg}
\centering
\begin{tabular}{c|l}
  \toprule
  $\theta$-dependence: & $\psi(\Lip_\delta g_1,\Lip_\delta g_2, u, v)=v \Lip_\delta g_2$ \\
  $\eta$-dependence: & $\psi(\Lip_\delta g_1,\Lip_\delta g_2, u, v)=u\Lip_\delta g_1+v \Lip_\delta g_2$ \\
  $\kappa$-dependence: & $\psi(\Lip_\delta g_1,\Lip_\delta g_2, u, v)=uv \Lip_\delta g_1\Lip_\delta g_2$ \\
  $\lambda$-dependence: & $\psi(\Lip_\delta g_1,\Lip_\delta g_2, u, v) = u\Lip_\delta g_1+v \Lip_\delta g_2 + uv \Lip_\delta g_1\Lip_\delta g_2$ \\
  \bottomrule
\end{tabular}
\end{table}

Later, in \cite{dedecker2004coupling} and \cite{dedecker2005new}, the authors introduced a new set of dependence measures that, instead of putting focus on the covariance structure, highlights the intrinsic ``coupling" property of the sequence. In this paper we will be focused on one important member in this family, the $\tau$-dependence. Consider a general probability space $(\Omega,\cF,\P)$ and a random variable $X$ taking value in a Polish space $(\mathcal{X}, \norm{\cdot}_{\cX})$ endowed with a norm $\norm{\cdot}_{\cX}$ and satisfying $\norm{\norm{X-x_0}_\cX}_{L_1}<\infty$ for some $x_0\in\cX$. Consider a $\sigma$-field $\cA\subset \cF$. The $\tau$-measure of dependence between $X$ and $\cA$ is defined to be
\begin{align*}
\tau(\mathcal{A}, X;\norm{\cdot}_{\cX}) = \Big\lVert\sup_{g \in \Lambda(\norm{\cdot}_{\cX})}\Big\{\int g(x)\P_{X|\mathcal{A}}({\sf d}x) - \int g(x)\P_X({\sf d}x) \Big\}\Big\lVert_{L_1},
\end{align*}
where $\P_X$ and $\P_{X|\mathcal{A}}$ represent the distribution and the conditional distributions of $X$ and $X$ given $\mathcal{A}$, $\Lambda(\norm{\cdot}_{\cX})$ stands for the set of 1-Lipschitz functions from $\mathcal{X}$ to $\R$ with respect to the norm $\norm{\cdot}_{\cX}$. 

The following theorem, extracted from \cite{dedecker2004coupling} and \cite{dedecker2007book}, characterizes the intrinsic ``coupling property" of $\tau$-measure of dependence and, as a matter of fact, gives an alternative definition of $\tau$-measure that is usually easier to use.

\begin{thm}[Lemma 3 in \cite{dedecker2004coupling}, Lemma 5.3 in \cite{dedecker2007book}]\label{thm:dedecker}
Let $(\Omega, \mathcal{\mathcal{F}}, \P)$ be a probability space, $\mathcal{A}$ be a $\sigma$-field of $\mathcal{F}$, and $X$ be a random variable with values in a Polish space $(\mathcal{X}, \norm{\cdot}_{\cX})$. If $Y$ is a random variable distributed as $X$ and independent of $\mathcal{A}$, then
\[
\tau(\mathcal{A}, X;\norm{\cdot}_{\cX}) \leq \E \norm{X-Y}_{\cX}.
\]
Assume that $\int \norm{x-x_0}_{\cX}\P_X({\sf d}x)$ is finite for any $x_0 \in \mathcal{X}$. Assume that there exists a random variable $U$ uniformly distributed over $[0,1]$, independent of the sigma-field generated by $X$ and $\mathcal{A}$. Then there exists a random variable $\tilde{X}$, measurable with respect to $\mathcal{A}\vee \sigma(X) \vee \sigma(U)$, independent of $\mathcal{A}$ and distributed as $X$, such that 
	\begin{align*}
	\tau(\mathcal{A}, X;\norm{\cdot}_{\cX}) = \E \norm{X-\tilde{X}}_{\cX}.
	\end{align*}
\end{thm}

We now apply the notion of $\tau$-dependence to a time series model. Let $\{X_j\}_{j \in J}$ be a set of $\mathcal{X}$-valued random variables with index set $J$ of finite cardinality. Then define
\begin{align*}
\tau(\mathcal{A}, \{X_j \in \mathcal{X}\}_{j \in J}; \norm{\cdot}_{\cX}) = \Big\lVert\sup_{g \in \Lambda(\norm{\cdot}_{\cX}')}\Big\{\int g(x)\P_{\{X_j\}_{j \in J}|\mathcal{A}}({\sf d}x) - \int g(x)\P_{\{X_j \}_{j \in J}}({\sf d}x) \Big\}\Big\lVert_{L_1},
\end{align*}
where $\P_{\{X_j \}_{j \in J}}$ and $\P_{\{X_j\}_{j \in J}|\mathcal{A}}$ represent the distribution of $\{X_j \}_{j \in J}$ and the conditional distribution of $\{X_j \}_{j \in J}$ given $\mathcal{A}$ respectively, and  $\Lambda(\norm{\cdot}_{\cX}')$ stands for the set of 1-Lipschitz functions:
\[
\Lambda(\norm{\cdot}_{\cX}'):= \Big\{f:\underbrace{\mathcal{X}\times \cdots \times \mathcal{X}}_{\card(J)}\to \R; f\text{ is 1-Lipschitz with respect to } \norm{\cdot}_{\cX}'\Big\}
\]
with $\norm{x}_{\cX}' := \sum_{j \in J}\norm{x_j}_{\cX}$ for any $x=(x_1,\ldots,x_J)\in\cX^{{\card}(J)}$. 

Using these concepts, for a time series $\{X_t\}_{t\in \Z}$, it is ready to define measure of temporal correlation strength as 
\begin{align*}
&\tau(\{X_t\}_{t \in \Z}; m, \norm{\cdot}_{\cX}) :=\\
&\quad\quad \sup_{i > 0}\max_{1 \leq \ell \leq i} \ell^{-1}\sup\Big\{\tau\{\sigma(X_{-\infty}^a), \{X_{j_1}, \dots, X_{j_\ell}\}; \norm{\cdot}_{\cX}\}, a+m \leq j_1 < \dots < j_\ell\Big\},
\end{align*}
where the inner supremum is taken over all $a \in \Z$ and all $\ell$-tuples $(j_1, \dots, j_\ell)$. 

In the end, let's consider $\{X_t\}_{t\in \Z}$ to be a real stationary causal process of the form
\begin{align}\label{eq:wu}
X_i=g(\cdots,\epsilon_{i-1},\epsilon_i),
\end{align}
with $\{\epsilon_i\}_{i\in\Z}$ an independent and identically distributed (i.i.d.) sequence and $g(\cdot)$ a measurable function such that the above time series model is properly defined. In \cite{wu2005nonlinear}, the author introduced the functional dependence measure, as manifested below.

\begin{definition}[Functional dependence measure, \cite{wu2005nonlinear}] \label{def:wu} Let $\{\epsilon_i,\epsilon_j'\}_{i,j\in\Z}$ be i.i.d. random variables. Let $X_m' := g(\cdots,\epsilon_{-2},\epsilon_{-1},\epsilon_0',\epsilon_1,\ldots,\epsilon_m)$. The functional dependence measure with regard to the $L_p$ norm is defined to be
\[
\theta_{m,p}:=\norm{X_m-X_m'}_{L_p}
\]
with the tail sum $\Theta_{m,p}:=\sum_{k=m}^{\infty}\theta_{k,p}$.
\end{definition}

The functional dependence measure $\theta_{m,p}$ is flexible and easy to compute in many applications; we refer the readers of interest to \cite{wu2011asymptotic} for a systematic review. In addition, given the data generating mechanism $g$, one can numerically compute functional dependence measures by Monte Carlo simulations. In contrast, numeric computation of other dependence measures can be highly nontrivial due to their definitions. 

We also mention a connection between physical dependence and $\tau$-dependence. As is apparent by comparing Theorem \ref{thm:dedecker} with Definition \ref{def:wu}, $\tau$-dependence and functional dependence measure are interestingly intrinsically connected. In particular, they are both adaptable to a notion of coupling. However, as noted in \citet[Remark 3.1]{dedecker2007book}, coupling in functional dependence is given in \cite{dedecker2004coupling} with all elements in the past, while in \cite{wu2005nonlinear} with only element in the past.

\section{Probability and moment inequalities under dependence} \label{sec:inequ}

Probability and moment inequalities play an important role in studying the statistical properties of estimators of parameters in statistical models. They are key in high-dimensional statistical theory, which is by its nature nonasymptotic. Of particular importance are those that give rise to efficient control of tail deviations, namely, higher-order moment and exponential-type inequalities. In this section we will give a brief review of some developed inequalities for time series, which are promising to be applied to the analysis of high-dimensional time series data. For this, this note is restricted to those built on the weak dependence measures introduced in Section \ref{sec:intro}, while those built on other structures like Markov chains or martingales, though related, shall not be covered. 

Before diving into the details, let's first fix what we mean a high-dimensional time series model. To characterize the impact of dimensionality on the performance of an estimator, it has become well-accepted in literature to model high-dimensional data under a triangular-array setting (see, for example, Section 1 in \cite{greenshtein2004persistence}). Applied to time series models, the following model will be used throughout the rest of this paper:
For each $n\in\N$, let $\{ X_{t,n}\}_{t\in\Z} $ denote a $d_n$-dimensional real time series with $d_n\in \N$ as well as the time series itself depending on $n$. For each  $n\in \N$, a length of $n$  fragment $\{X_{i,n}\}_{i\in[n]}$, with $[n]:=\{1,2,\ldots,n\}$, is observed from the time series $\{ X_{t,n}\}_{t\in\Z} $.  For different $n$, a different  time series with possibly different dimension is observed. As $n$ goes to infinity, the dimension of the $n$ fragment time series, $d_n$, is allowed to increase to infinity as well.  

To name one particular example, let's consider the observations $\{ X_{t,n}\}_{t\in[n]} $ to be generated from a VAR(1) model, $\cM_n$, that is changing with $n$:
\[
\cM_n: \Big\{\{X_{t,n}\}_{t\in\Z}: X_{t,n}=A_nX_{t-1,n}+E_{t,n}, ~{\rm for~all~}~t\in\Z\Big\}.
\]
Here  $A_n$  is  a  $d_n \times d_n$-dimensional transition matrix, $E_{t,n}$ is  a $d_n$-dimensional vector of error term.  The value $A_{n}$ and the dimension $d_n$ are both allowed to change  with $n$; e.g., it could be true that
\begin{itemize}
\item[] As $n=1$, a  $1$-dimensional, length of $1$ fragment, $\{X_{1,1}\}$,  is observed from the model $\cM_1$ with $A_1=0.5$;
\item[] As $n=2$, a $2$-dimensional, length of $2$ fragment time series, $\{X_{1,2},X_{2,2}\}$, is observed from the model $\cM_2$ with $$A_2=\begin{pmatrix} 
0.5 & 0.1 \\
0.2 & 0.25
\end{pmatrix};$$
\item[] As $n=3$, a $4$-dimensional, length of $3$ fragment time series, $\{X_{1,3},X_{2,2}, X_{3,3}\}$, is observed from the model $\cM_3$ with $$A_3=\begin{pmatrix} 
0.2 & 0 & 0.1 & 0.4 \\
0.2 & 0.1 & 0.1 & 0.2 \\
0.1 & 0.2 & 0.3 & 0.1 \\
0 & 0 & 0 & 0.1
\end{pmatrix};$$
\item[] $\cdots\cdots$.
\end{itemize}

\subsection{Sample sum for scalars with $d_n=1$}

Several of the most essential moment inequalities are surrounding the sample sum. In detail, for any $n\in\N$, consider a time series $\{X_{t,n}\}_{t\in\Z}$ and its size-$n$ fragment $\{X_{i,n}\}_{i\in[n]}$. Our aim is to characterize the moment and tail properties for $\sum_{i=1}^n \Big(X_{i,n}-\E X_{i,n}\Big)$. Without loss of generality, in the following it is assumed that the time series has margin mean-zero. In this section we are focused on the sample sum $S_n:=\sum_{i=1}^nX_{i,n}$ of fixed dimension $d_n=1$; in the later sections we shall allow $d_n$ to increase to infinity. 

To start with, let's first consider the case of linear processes by assuming that $\{X_{t,n}\}_{t\in\Z}$ follows a linear process
\begin{align}\label{eq:linear}
X_{t,n}=\sum_{j=0}^{\infty}f_{j,n}\epsilon_{i-j,n},
\end{align}
with $\{\epsilon_{j,n}\}_{j\in\Z}$ understood to be an i.i.d. scalar sequence with mean zero and $\norm{\epsilon_{0,n}}_{L_p}<\infty$ for some $p>2$, and $f_n:=\{f_{j,n}\}$ as a real coefficient sequence satisfying $\norm{f_n}_2^2:=\sum_{j=0}^{\infty}f_{j,n}^2<\infty$. The form \eqref{eq:linear} is very general and includes many famous time series models such as the ARMA processes. 

The first result concerns such time series of the particular form \eqref{eq:linear}, and is from \cite{wu2016performance}. It gives a Nagaev-type inequality for linear processes, including both short- and long-range dependence cases. 

\begin{thm}[Theorem 1, \cite{wu2016performance}] Assume the linear process in \eqref{eq:linear}. 
\begin{itemize}
\item[(i)] (Short-range dependence) Let $c_p:=2e^{-p}(p+2)^{-2}$. If $\norm{f_{n}}_1:=\sum_{j=0}^{\infty}|f_{j,n}|<\infty$, then for any $x>0$ we have
\[
\P(|S_n|\geq x) \leq \Big(1+\frac{2}{p}\Big)^p\cdot\frac{n\norm{f_n}_1^p\norm{\epsilon_{0,n}}_{L_p}^p}{x^p}+2\exp\Big(-\frac{c_px^2}{n\norm{f_n}_1^2\norm{\epsilon_{0,n}}_{L_2}^2}  \Big).
\]
\item[(ii)] (Long-range dependence) Assume $K_n:=\sup_{j\geq 0}|f_{j,n}|(1+j)^\beta<\infty$ for some $1/2<\beta<1$. Then there exist  constants $C_1,C_2$ only depending on $p$ and $\beta$ such that, for all $x>0$,
\[
\P(|S_n|\geq x)\leq C_1\frac{n^{1+p(1-\beta)}K_n^p\norm{\epsilon_{0,n}}_{L_p}^p}{x^p}+2\exp\Big(-\frac{C_2x^2}{n^{3-2\beta}\norm{\epsilon_{0,n}}_{L_2}^2K_n^2}  \Big).
\]
\end{itemize}
\end{thm}

We then move on to the general possibly nonlinear case. The first of such results considers the $\phi$-mixing case and is from \cite{dedecker2005new}.

\begin{thm}[Proposition 5, \cite{dedecker2005new}] Let $\{X_{t,n}\}_{t\in\Z}$ be a mean-zero stationary sequence of dimension $d_n$ fixed to be 1. Let $\phi_n(m):=\phi(\{X_{t,n}\}_{t\in\Z};m)$ and $|X_{0,n}|\leq C_n$ for some constant $C_n$ that possibly depends on $n$. Then, for every $p=2,3,\ldots$ and any $n\geq 1$, the following inequality holds:
\[
\E|S_n|^p \leq \Big(8C_n^2p\sum_{i=0}^{n-1}(n-i)\phi_n(i)\Big)^{p/2}.
\]
\end{thm}

The next result considers the $\alpha$- and $\tau$-mixing cases and is from \cite{merlevede2009bernstein}.

\begin{thm}[Theorem 2, \cite{merlevede2009bernstein}]
\label{lemma:sample_mean_Bern}
Let $\{X_{t,n}\}_{t\in\Z}$ be a stationary mean-zero sequence of dimension $d_n$ fixed to be 1. Suppose that the sequence satisfies either a geometric $\alpha$-mixing condition:
\begin{align*}
\alpha(\{X_{t,n}\}_{t\in\Z};m) \leq \textup{exp}(-\gamma_n m),~~{\rm for~}m=1,2,\ldots
\end{align*}
or a geometric $\tau$-mixing condition:
\begin{align*}
\tau(\{X_{t,n}\}_{t\in\Z};m,|\cdot|) \leq \textup{exp}(-\gamma_n m),~~{\rm for~}m=1,2,\ldots
\end{align*}
with some positive constant $\gamma_n$ that could depend on $n$, and there exists a positive constant $B_n$ such that $\sup_{i\geq 1}\|X_{i,n}\|_{L_\infty} \leq B_n$. Then there are positive constants $C_{1,n}$ and $C_{2,n}$ depending only on $\gamma_n$ such that for all $n\geq 2$ and positive $t$ satisfying $t < 1/[C_1B(\log n)^2]$, the following inequality holds:
\begin{align*}
\log[\E\exp(tS_n)] \leq \frac{C_{2,n}t^2(n\sigma_n^2+B_n^2)}{1-C_{1,n}tB_n(\log n)^2},
\end{align*}
where $\sigma_n^2$ is defined by 
\begin{align*}
\sigma_n^2 := {\rm Var}(X_{1,n}) + 2\sum_{i>1}\Big|{\rm Cov}(X_{1,n},X_{i,n})\Big|.
\end{align*}
\end{thm}

We note here that the dependence of $C_{1,n}$ and $C_{2,n}$ on $\gamma_n$ could be explicitly calculated, as have been made in \cite{banna2016bernstein} and \cite{han2018moment}; also refer to the later Theorems \ref{thm:matrix1} and \ref{thm:matrix2}.

The next result considers the weak dependence case, and is the foundation of dependence measures proposed in \cite{doukhan1999new}. We refer to \cite{doukhan1999new} and \cite{Doukhan2007} for the relation between those weak dependences defined in Definition \ref{def:weakdependence} and the following Equations \eqref{ineq:CovCond} and \eqref{ineq:SumRho}.

\begin{thm}[a slight modification to Theorem 1 in \cite{Doukhan2007}]
Suppose $\{X_{i,n}\}_{i\in[n]}$ are real-valued random variables with mean 0, defined on a common probability space $(\Omega, \mathcal{A}, \P)$. Let $\Psi: \N^2 \rightarrow \N$ be one of the four functions defined in Table \ref{tab:WDeg}. Assume that there exist constants $K_n, M_n, L_{1,n}, L_{2,n} >0$, $a_n,b_n \geq 0$, and a nonincreasing sequence of real coefficients $\{\rho_n(i)\}_{i \geq 0}$ such that for any $u$-tuple $(s_1,\ldots, s_u)$ and $v$-tuple $(t_1,\ldots,t_v)$ with $1\leq s_1\leq \cdots \leq s_u < t_1\leq \cdots \leq t_v\leq n$, we have
\begin{align}
&\Bigl|\Cov\Bigl(\prod_{i=1}^u X_{s_i,n},\prod_{j=1}^v X_{t_j,n}\Bigr)\Bigr| \leq K_n^2 M_n^{u+v} \{(u+v)!\}^{b_n} \Psi(u,v)\rho_n(t_1-s_u), \label{ineq:CovCond} 
\end{align}
where the sequence $\{\rho_n(i)\}_{i \geq 0}$ satisfies
\begin{align}
&\sum_{s=0}^\infty (s+1)^k\rho_n(s) \leq L_{1,n}L_{2,n}^k(k!)^{a_n},~ \mbox{for any}~ k\in \N. \label{ineq:SumRho}
\end{align}
Moreover, we require that the following moment condition holds:
\begin{align*}
&\E |X_{i,n}|^k \leq (k!)^{b_n} M_n^k,~ i=1,\ldots,n, ~ \mbox{for any}~ k\in \N.
\end{align*}
Then, for any $n\geq 1$ and any $x>0$, we have
\begin{align*}
\P(S_n \geq x)\leq \exp\Bigl\{ -\frac{x^2}{C_{1,n}n + C_{2,n} x^{(2a_n+2b_n+3)/(a_n+b_n+2)}} \Bigr\},
\end{align*}
where $C_{1,n}$ and $C_{2,n}$ are constants that can be chosen to be
\begin{align*}
&C_{1,n}=2^{a_n+b_n+3}K_n^2M_n^2L_{1,n}(K_n^2\vee 2),~~ C_{2,n}=2\{M_nL_{2,n}(K_n^2\vee 2)\}^{1/(a_n+b_n+2)}.
\end{align*}
\end{thm}

\begin{proof}
The proof follows that of Theorem 1 in \cite{Doukhan2007} with minor modifications, as listed below. Restricted to this proof, we inherit the notation in \cite{Doukhan2007} and abandon the subscript $n$.

Equation (30) in \cite{Doukhan2007} can be strengthened to
\begin{align*}
\E |Y_j| \leq 2^{k-j-1}\{(k-j+1)!\}^b K^2M^k\rho(t_{i+1}-t_i).
\end{align*}
This leads to
\begin{align}\label{ineq:Lem13*}
|\overline{\E}(X_{t_1}\cdots X_{t_k})| \leq 2^{k-1} (k!)^b K^2 M^k \rho(t_{i+1}-t_i),
\end{align}
which corresponds to Lemma 13 in \cite{Doukhan2007}. Using \eqref{ineq:Lem13*}, we obtain that
\begin{align*}
\Bigl|\Gamma(X_{t_1},\ldots,X_{t_k})\Bigr| \leq& \sum_{\nu=1}^k \sum_{\bigcup_{p=1}^{\nu} I_p=I} N_{\nu} (I_1,\ldots,I_{\nu}) 2^{k-\nu} (k!)^b K^{2\nu}M^k \min_{1\leq i < k} \rho(t_{i+1}-t_i)\\
\leq& K^2(K^2\vee 2)^{k-1} M^k (k!)^b \{(k-1)!\} \min_{1\leq i < k} \rho(t_{i+1}-t_i).
\end{align*}
Thus, we have
\vspace{-10pt}
\begin{align}\label{ineq:Lem14*}
\Bigl|\Gamma_k(S_n)\Bigr| \leq nK^2(K^2\vee 2)^{k-1} M^k (k!)^{b+1} \sum_{s=0}^{n-1}(s+1)^{k-2}\rho(s).
\end{align}
Equation \eqref{ineq:Lem14*} corresponds to Lemma 14 in \cite{Doukhan2007}. The rest follows the same technique as in \cite{Doukhan2007}.
\end{proof}

Lastly we consider the functional dependence setting. The first result is a Rosenthal-type inequality, and is from \cite{liu2013probability}.

\begin{thm}[Theorem 1, \cite{liu2013probability}]\label{thm:wu1} Assume $\{X_{t,n}\}_{t\in\Z}$ is of dimensional $d_n=1$ and is generated from the model \eqref{eq:wu} with functional dependence measures $\theta_{m,p,n}$, which is of an additional subscript $n$ to highlight its dependence on $n$. Assume further that $\E X_{0,n}=0$, $\E|X_{0,n}|^p<\infty$, and $p>2$. Then we have, for any $n\geq 1$,
\begin{align*}
\norm{S_n}_{L_p} \leq &n^{1/2}\Big[ \frac{87p}{\log p}\sum_{j=1}^n\theta_{j,2,n}+3(p-1)^{1/2}\sum_{j=n+1}^\infty \theta_{j,p,n}+\frac{29p}{\log p}\norm{X_{0,n}}_{L_2}\Big]\\
&+n^{1/p}\Big[\frac{87p(p-1)^{1/2}}{\log p}\sum_{j=1}^n j^{1/2-1/p}\theta_{j,p,n}+\frac{29p}{\log p}\norm{X_{0,n}}_{L_p}  \Big].
\end{align*}
\end{thm}

The second is a Nagaev-type inequality, and is also from \cite{liu2013probability}.

\begin{thm}[Theorem 2, \cite{liu2013probability}]\label{thm:wu2} Assume $\{X_{t,n}\}_{t\in\Z}$ is of dimensional $d_n=1$ and is generated from the model \eqref{eq:wu} with functional dependence measures $\theta_{m,p,n}$. Assume further that $\E X_{0,n}=0$, $\E|X_{0,n}|^p<\infty$, and $p>2$. Then we have the following bounds for any $n\geq 1$.
\begin{itemize}
\item[(i)] Denote 
\[
\mu_{j,n}:=(j^{p/2-1}\theta_{j,p,n}^p)^{1/(p+1)}~~~{\rm and }~~\nu_n:=\sum_{j=1}^{\infty}\mu_{j,n}<\infty.
\]
Then, for any $x>0$,
\begin{align*}
\P(|S_n|\geq x)\leq c_p\frac{n}{x^p}\Big(\nu_n^{p+1}+\norm{X_{0,n}}_{L_p}^p\Big)+4\sum_{j=1}^{\infty}\exp\left(-\frac{c_p\mu_{j,n}^2x^2}{n\nu_n^2\theta_{j,2,n}^2} \right) + 2\exp\left(-\frac{c_px^2}{n\norm{X_{0,n}}_{L_2}^2} \right),
\end{align*}
where $c_p>0$ is a constant only depending on $p$.
\item[(ii)] Assume that $\Theta_{m,p,n}:=\sum_{k=m}^{\infty}\theta_{k,p,n}=O(m^{-\alpha})$ as $m$ goes to infinity, with some constant $\alpha>1/2-1/p$. Then there exist absolute positive constants $C_1,C_2$ such that, for any $x>0$,
\[
\P(|S_n|\geq x)\leq \frac{C_1\Theta_{0,p,n}^pn}{x^p}+4G_{1-2/p}\left(\frac{C_2x}{\sqrt{n}\Theta_{0,p,n}}\right),
\]
where for any $y>0$, $q>0$, $G_q(y)$ is defined to be
\[
G_q(y)=\sum_{j=1}^{\infty}\exp(-j^qy^2).
\]
\item[(iii)] If $\Theta_{m,p,n}=O(m^{-\alpha})$ as $m$ goes to infinity, with some constant $\alpha<1/2-1/p$, then
\[
\P(|S_n|\geq x)\leq \frac{C_1\Theta_{0,p,n}^pn^{p(1/2-\alpha)}}{x^p}+4G_{(p-2)/(p+1)}\left(\frac{C_2x}{n^{(2p-1-2\alpha p)/(2+2p)}\Theta_{0,p,n}}\right).
\]
\end{itemize}
\end{thm}

It should be noted that Theorems \ref{thm:wu1} and \ref{thm:wu2} actually apply to cases beyond the sample sum, and the same inequalities hold for the partial sum process $S_n^*:=\max_{1\leq k\leq n}|\sum_{i=1}^k X_{i,n}|$. However, if $S_n$, instead of $S_n^*$, is of interest, Theorem \ref{thm:wu2} could be further strengthened, as was made in \cite{wu2016performance}. To this end, let's first introduce the dependence adjusted norm (DAN) for the process $\{X_{t,n}\}_{t\in\Z}$ as 
\begin{equation}
    \label{eq:dan}
\norm{X_{\cdot, n}}_{p,\alpha}:= \sup_{m\geq 0}(m+1)^{\alpha}\Theta_{m,p,n}.
\end{equation}

\begin{thm}\label{th:ww}[Theorem 2, \cite{wu2016performance}] Assume $\{X_{t,n}\}_{t\in\Z}$ is of dimensional $d_n=1$ and is generated from the model \eqref{eq:wu} with functional dependence measures $\theta_{m,p,n}$. Assume further that $\E X_{0,n}=0$ and $\norm{X_{\cdot,n}}_{p,\alpha}<\infty$ for some $p>2$ and $\alpha>0$. Let
\[
a_n=\begin{cases}
1,  & ~~~\text{if }\alpha>1/2-1/p,\\
n^{p/2-1-\alpha p},&~~~\text{if }\alpha<1/2-1/p.
\end{cases}
\]
Then there exists constants $C_1,C_2,C_3$ only depending on $q$ and $\alpha$ such that, for all $x>0$, we have
\[
\P(|S_n|\geq x)\leq C_1\frac{a_nn\norm{X_{\cdot,n}}_{p,\alpha}^p}{x^p} + C_2\exp\Big(-\frac{C_3x^2}{n\norm{X_{\cdot,n}}_{2,\alpha}^2}  \Big).
\]
\end{thm}

\subsection{Sample sum for random vectors with $d_n \ge 1$}
In this section we will concern the sample sum case when $d = d_n$ potentially diverges to infinity with $n$. Let $\{X_{t,n}\}_{t\in\Z}$ be a $d_n$-dimensional real time series of the form (\ref{eq:wu}):
\begin{align}\label{eq:vwu}
X_{t,n}=g_n(\cdots,\epsilon_{t-1,n},\epsilon_{t,n})=
\begin{pmatrix}
g_{1,n}(\cdots,\epsilon_{t-1,n},\epsilon_{t,n})\\
g_{2,n}(\cdots,\epsilon_{t-1,n},\epsilon_{t,n})\\
\vdots\\
g_{d_n,n}(\cdots,\epsilon_{t-1,n},\epsilon_{t,n})
\end{pmatrix}
.
\end{align}
Assume $\E X_{t,n} = 0$ and let $S_n = \sum_{i=1}^n X_{i,n}$. Theorem \ref{wz17} below provides a tail probability for $|S_n|_\infty$, where for any vector $v=(v_1,\ldots,v_d)^\top$ let $|v|_\infty=\max_{j\in[d]}|v_j|$. Assume $\E |X_{i,n}|^q < \infty, q > 2$, and define the uniform functional dependence measure
\begin{equation}
    \delta_{i, q,n} = \| |X_{i,n} - X_{i, \{0\},n}|_\infty \|_{L_q},
\end{equation}
where
\begin{equation}
     |X_{i,n} - X_{i, \{0\},n}|_\infty := \max_{j \le d} \Big|g_{j,n}(\cdots,\epsilon_{i-1,n},\epsilon_{i,n})- g_{j,n}(\cdots,\epsilon_{-2,n},\epsilon_{-1,n},\epsilon_0',\epsilon_{1,n},\ldots,\epsilon_{i,n})\Big|.
\end{equation}
Define the vector version DAN (cf. (\ref{eq:dan})) as 
\begin{equation}
    \| |X_{.,n}|_\infty \|_{q, \alpha} = \sup_{m \ge 0} (m+1)^\alpha \Omega_{m, q,n}, 
    \mbox{ where } \Omega_{m, q,n} = \sum_{i=m}^\infty \delta_{i, q,n}.
\end{equation}
The constants $C_{q,\alpha} > 0$ in Theorem \ref{wz17} only depend on $q$ and $\alpha$ and their values may change from place to place.

\begin{thm}\label{wz17}[Theorem 6.2 in \cite{zhang2017gaussian}]
Assume $ \| |X_{.,n}|_\infty \|_{q, \alpha} < \infty$, where $q > 2, \alpha > 0$. Let $\Psi_{2, \alpha,n} = \max_{j \le d} \| X_{\cdot j,n} \|_{q, \alpha}$ be the counterpart of $ \| |X_{.,n}|_\infty \|_{q, \alpha}$ with the maximum over $j\in[d_n]$ taken outside instead of inside the expectation. Let $\ell_n = \max(1, \log d_n)$. (i) If $\alpha > 1/2 - 1/q$, then for all $x \ge C_{q,\alpha} (\sqrt{n \ell_n} \Psi_{2, \alpha,n} + n^{1/q} \ell_n^{3/2} \| |X_{.,n}|_\infty \|_{q, \alpha} )$, we have 
\begin{align*}
\P\Big\{ |S_n |_\infty \geq  x \Big\} \leq  C_{q, \alpha} { {n \ell_n^{q/2} \| |X_{.,n}|_\infty \|_{q, \alpha}^q} \over x^q} +
C_{q, \alpha} \exp\Big(- C_{q,\alpha} {x^2 \over {n \Psi_{2, \alpha,n} ^2}}\Big).
\end{align*}
(ii) If $\alpha < 1/2 - 1/q$, then for all $x \ge C_{q,\alpha}(\sqrt{n \ell_n} \Psi_{2, \alpha,n} + n^{1/2-\alpha} \ell_n^{3/2} \| |X_{.,n}|_\infty \|_{q, \alpha} ) $, we have 
\begin{align*}
\P\Big\{ |S_n |_\infty \geq x \Big\} \leq  C_{q, \alpha} { {n^{q/2-\alpha q} \ell_n^{q/2} \| |X_{.,n}|_\infty \|_{q, \alpha}^q} \over x^q} +
C_{q, \alpha} \exp\Big(- C_{q,\alpha} {x^2 \over {n \Psi_{2, \alpha,n} ^2}}\Big).
\end{align*}
\end{thm}

\begin{example}\label{eq:LLP}
As an application of Theorem \ref{wz17}, consider the following example, with the subscript $n$ omitted for presentation simplicity. Let $W_i = \sum_{j=0}^\infty a_j \epsilon_{i-j}$ be a linear process, where $\epsilon_j$ are i.i.d. innovations with finite $q$th moment $\mu_q := \|\epsilon_i \|_{L_q} < \infty$, $q > 2$, and $a_j$ are coefficients satisfying $a_* := \sup_{m \ge 0} (m+1)^\alpha \sum_{i=m}^\infty |a_i| < \infty$. Let $X_{i j} = g_j(W_i) - E  g_j(W_i) $, where $g_j$ are Lipschitz continuous functions with constants bounded by $L$.  Then $|X_i - X_{i, \{0\}}|_\infty \le L |a_i| | \epsilon_0 - \epsilon_0'|$ and $\delta_{i, q} \le 2 L |a_i| \mu_q $. The dependence adjusted norms $\| X_{\cdot j} \|_{q, \alpha} \le 2 L \mu_q a_*$ and $\| |X_.|_\infty \|_{q, \alpha} \le 2 L \mu_q a_*$. In comparison with Theorem \ref{th:ww}, the bound in Theorem \ref{wz17} is sharp up to a multiplicative logarithmic factor $(\log d)^{q/2}$, adjusting for multi-dimensionality. 
\end{example}

\begin{example}
(Largest eigenvalues of sample auto-covariance matrices) Let $W_i$ in Example \ref{eq:LLP} be of the form of stationary causal process (\ref{eq:wu}) with $\E W_i = 0$, $\E |W_i|^q < \infty$, $q > 2$. Again let's omit the subscript $n$ for no confusion will be made. Let $a_i = \| W_i - W_i'\|_{L_q}$ be the associated functional dependence measure, and assume the dependence adjusted norm $\| X_\cdot\|_{q, \alpha} < \infty$, $\alpha > 1/2 - 1/q$. Let $S_n(\theta) = \sum_{t=1}^n W_t \exp(\sqrt{-1} t \theta)$, $0 \le \theta \le 2 \pi$, be the Fourier transform of $(W_t)_{t=1}^n$, where $\sqrt{-1}$ is the imaginary unit. Let $d = n^9$ and $\lfloor \theta \rfloor_d = 2 \pi \lfloor d \theta / (2 \pi) \rfloor /d$. By Theorem \ref{wz17}(i), the inequality therein holds with $\max_{0 \le \theta \le 2 \pi} |S_n (\lfloor \theta \rfloor_d) | = \max_{j \le d} |S_n(2\pi j/d)|$. Noting that $\|\max_\theta |S_n (\lfloor \theta \rfloor_d) - S_n (\theta)| \|_q \le \|W_1\|_q / n^6$. Thus with elementary manipulations the same inequality in Theorem \ref{wz17}(i) holds with $\max_{0 \le \theta \le 2 \pi} |S_n (\theta) |$. 

Given $(W_t)_{t=1}^n$, let the sample covariance matrix 
\begin{equation*} \hat \Sigma_n = (\hat \gamma_{j-k}), \mbox{ where }
\hat \gamma_k = n^{-1} \sum_{l=k+1}^n W_l W_{l-k}, ~~0 \le k \le n-1,
\end{equation*}
Notice that the largest eigenvalues
\begin{equation*}
 \lambda_{\max} (\hat \Sigma_n) \le \max_{0 \le \theta \le 2 \pi} |S_n (\theta) |^2 / n.
\end{equation*}
We obtain the tail probability inequality
\begin{eqnarray*}
    \P\Big( \lambda_{\max} (\hat \Sigma_n)  \ge u\Big) 
    &\le& \P\Big( \max_{0 \le \theta \le 2 \pi} |S_n (\theta) | \ge (n u)^{1/2}\Big) \cr
    & \leq&  C_{q, \alpha} { {n (\log n)^{q/2} \| X_. \|_{q, \alpha}^q} \over (n u)^{q/2}} +
C_{q, \alpha} \exp\Big(- C_{q,\alpha} {u \over {\| X_. \|_{2, \alpha}^2 }}\Big),
\end{eqnarray*}
when $u \ge C_{q, \alpha} \| X_. \|_{2, \alpha} ^2 \log n$ for a sufficient large constant $C_{q, \alpha}$.
\end{example}

\subsection{Sample sum for random matrices with $d_n\geq 1$}

In this section we will consider the case of time dependent random matrices. Here $X_{t,n}\in \reals^{d_n\times d_n}$ is a $d_n$-dimensional random matrix and $\{X_{t,n}\}_{t\in\Z}$ is a matrix-valued time series. Tail probability inequalities for spectral norms for the sum $\sum_{t=1}^n X_{t, n}$ will be presented. The latter results are useful for statistical inference for convariance matrices of high dimensional time series.  



\cite{ahlswede2002strong}, \cite{oliveira2009concentration}, \cite{tropp2012user}, among many others, have studied such bounds when $\{X_{t,n}\}_{t\in\Z}$ are mutually independent. For instance, \cite{oliveira2009concentration} and \cite{tropp2012user} have introduced the following Bernstein-type inequality for tails. The result in \cite{oliveira2009concentration} also applies to martingales (cf. Freedman's Inequality for matrix martingales \citep{freedman1975tail}). 	Also see \cite{mackey2014matrix} for further extensions to conditionally independent sequences and combinatorial sums. 
\begin{thm}[Corollary 7.1 in \cite{oliveira2009concentration}, Theorem 1.4 in \cite{tropp2012user}]
Let $X_{1,n}, \dots, X_{n,n}$ be real, mean-zero, symmetric independent $d_n \times d_n$ random matrices and assume there exists a positive constant $M_n$ such that $\lambda_{\max}(X_{i,n}) \leq M_n$  for all $1 \leq i \leq n$. Then for any $x \geq 0$, 
\begin{align*}
\P\Big\{ \lambda_{\max}\Big( \sum_{i = 1}^n X_{i,n}\Big) \geq x \Big\} \leq d_n\exp\Big(-\frac{x^2}{2\sigma_n^2 + 2M_nx/3}\Big),
\end{align*}
where $\sigma_n^2 := \lambda_{\max}( \sum_{i = 1}^n \E X_{i,n}^2)$ and recall that $\lambda_{\max}(\cdot)$ represents the largest eigenvalue of the input.
\end{thm}

Assuming $\{X_{t,n}\}_{t\in\Z}$ satisfies a geometrically $\beta$-mixing decaying rate:
\begin{align}\label{eq:banna-mixing}
\beta(\{X_{t,n}\}_{t\in\Z};m) \leq \exp\{-\gamma_n(m-1)\},~~{\rm for~}m=1,2,\ldots
\end{align}
with some constant $\gamma_n>0$ possibly depending on $n$, \cite{banna2016bernstein} proved the following theorem that extends the matrix Bernstein inequality to the $\beta$-mixing case. In the sequel, for any set $A$, we denote $\card{(A)}$ to be its cardinality.

\begin{thm}[Theorem 1 in \cite{banna2016bernstein}]\label{thm:matrix1}
	Let $\lbrace X_{t,n}\rbrace_{t \in \Z}$ be a sequence of mean-zero symmetric $d_n \times d_n$ random matrices with $\sup_{i\in[n]}\lambda_{\max}(X_{i,n}) \leq M_n$ for some positive constant $M_n$. Further assume \eqref{eq:banna-mixing} holds. Then there exists a universal positive constant $C$ such that, for any $n\geq 2$ and $x>0$, 
\begin{align*}
\P\Big\{ \lambda_{\max}\Big( \sum_{i = 1}^n X_{i,n}\Big) \geq x \Big\} \leq d_n\exp\Big\{-\frac{Cx^2}{\nu_n^2n + M_n^2/\gamma_n + xM_n\tilde{\gamma}(\gamma_n,n)} \Big\},
\end{align*}
where
\begin{align*}
\nu^2 := \sup\limits_{K\subset[n]}\frac{1}{\card(K)}\lambda_{\max}\Big\{\E \Big(\sum_{i \in K} X_{i,n}\Big)^2\Big\}~~{\rm and}~~\tilde{\gamma}(\gamma_n,n) := \frac{\log n}{\log 2}\max\Big(2, \frac{32\log n}{\gamma_n\log 2}\Big).
\end{align*}
\end{thm}

Later, this result is further extended to the $\tau$-mixing case, which was made in \cite{han2018moment}.

	\begin{thm}[Theorem 4.3 in \cite{han2018moment}]\label{thm:matrix2}
		Consider a sequence of real, mean-zero, symmetric $d_n \times d_n$ random matrices $\lbrace X_{t,n} \rbrace_{t\in \Z}$ with $\sup_{i\in[n]}\lVert X_{i,n} \rVert \leq M_n$ for some positive constant $M_n$ that is allowed to depend on $n$ and $\norm{\cdot}$ represents the matrix spectral norm. In addition, assume that this sequence is of a geometrically decaying $\tau$-mixing rate, i.e.,
		\[
		\tau(\{X_{t,n}\}_{t \in \Z};m, \norm{\cdot}) \leq M_n\psi_{1,n}\exp\{-\psi_{2,n}(m-1)\},~~{\rm for~}m=1,2,\ldots 
		\]
		with some constants $\psi_{1,n}, \psi_{2,n} > 0$. Denote $\tilde{\psi}_{1,n} := \max\{d_n^{-1}, \psi_{1,n}\}$. Then for any $ x\geq 0$ and any $n\geq 2$, we have 
		\begin{align*}
		\P\bigg\{\lambda_{\max}\bigg(\sum_{i = 1}^{n}X_{i,n}\bigg) \geq x\bigg\}
		\leq d_n \exp\bigg\{-\frac{x^2}{8(15^2n\nu_n^2 + 60^2M_n^2/\psi_{2,n}) + 2xM_n\tilde{\psi}(\tilde{\psi}_{1,n} ,\psi_{2,n},n,d_n)}\bigg\},
		\end{align*}
		where 
		\begin{align*}
		&\nu_n^2 := \sup_{K\subset [n]}\frac{1}{\card(K)}\lambda_{\max}\bigg\{\E\bigg(\sum_{i \in K}X_{i,n}\bigg)^2\bigg\}\\ 
		\mbox{and}~~~& \tilde{\psi}(\tilde{\psi}_{1,n} ,\psi_{2,n},n,d_n) := \frac{\log n}{\log 2} \max\bigg\{1,\frac{8\log(\tilde{\psi}_{1,n} n^6d_n)}{\psi_{2,n}}\bigg\}.
		\end{align*}
	\end{thm}

We note that the above matrix Bernstein inequalities for weakly dependent data can be immediately applied to study the behavior of many statistics of importance in analyzing a high-dimensional time series model. In particular, tail behaviors for the largest eigenvalues of sample autocovariances in weakly dependent high dimensional time series models have been characterized in \citet[Theorems 2.1 and 2.2]{han2018moment}, with bounds delivered for both general and Gaussian weakly dependent time series (the later using a different set of techniques tailored for Gaussian processes) separately.


\subsection{U- and V-statistics}

Consider $\{X_{i,n}\}_{i\in[n]}$ to be $n$ random variables of identical distribution in a measurable space $(\cX,\cB_{\cX})$. Given a symmetric kernel function $h_n(\cdot):\cX^r \rightarrow \reals$, the U-and V-statistic $U_n(h_n)$ and $V_n(h_n)$ of order $r_n$ are defined as:
\begin{align*}
&U_n(h_n):=\binom{n}{r_n}^{-1}\sum_{1\leq i_1<\cdots<i_{r_n}\leq n}h_n(X_{i_1,n},\ldots,X_{i_{r_n},n})\\
{\rm and}~~~&V_n(h_n):= n^{-r_n}\sum_{i_1, \ldots, i_{r_n} =1}^n h_n(X_{i_1,n},\ldots, X_{i_{r_n},n}).
\end{align*}
The V- and U-statistics are popular alternatives to sample sums and have been routinely used in statistics nowadays (cf. the textbooks \cite{lee1990u} and \cite{korolyuk1994theory}).  

Non-asymptotic probability and moment inequalities for V- and U-statistics in the i.i.d. case have been extensively studied \citep{hoeffding1963probability,arcones1993limit,gine2000exponential,adamczak2006moment}. 
Assumed $\{X_{t,n}\}_{t\in\Z}$ to be geometrically $\phi$-mixing, \cite{han2018exponential} established the following theorem that gives an exponential inequality for dependent U-statistics. 

\begin{thm}
\label{thm:212}
[Theorem 2.1, \cite{han2018exponential}] Let $\{X_{t,n}\}_{t\in\Z}$ satisfies
\[
\phi(\{X_{t,n}\}_{t\in\Z}; m)\leq c_{n}\exp(-C_{n}m)~~{\rm for~}m=1,2,\ldots
\]
with two constants $c_{n},C_{n}>0$. Assume further that $\norm{h_n}_{\infty}\leq M_n$, symmetric, and is mean-zero (i.e., $\E h_n=0$ with regard to the product measure). We then have, there exist two constants $c_{n}', C_n'>0$ that only depend on $c_n, C_n,$ and $r_n$, such that, for any $x\geq 0$ and $n\geq 4$,
\begin{align*}
\P(|U_n(h_n)|\geq c_n'M_n/\sqrt{n}+x)\leq2\exp\Big(-\frac{C_n'x^2n}{M_n^2+M_nx(\log n)(\log\log 4n)}\Big).
\end{align*}
\end{thm}

With tedious calculations, the dependence of $c_n',C_n'$ on $r_n, c_n,C_n$ in Theorem \ref{thm:212} can be explicitly obtained, as was made in Theorems \ref{thm:matrix1} and \ref{thm:matrix2}. 

In order to present the next result, let's first introduce more concepts in U- and V-statistics. For presentation clearness, let's assume the kernel $h_n(\cdot)$, its order $r_n$, and the dimension $d_n$ are fixed, and hence written as $h(\cdot)$, $r$, and $d$ without the subscript. Assume $\{X_{t,n}\}_{t\in\Z}$ to be stationary for any $n\in\N$. Let $\{\tilde{X}_{i,n}\}_{i\in[n]}$ be an i.i.d. sequence with $\tilde{X}_{1,n}$ identically distributed as $X_{1,n}$. The mean value of a symmetric kernel $h$ (with regard to the marginal probability measure $\P_n$) is defined as
\begin{align*}
\theta_n := \theta_n(h) := \E h(\tilde{X}_{1,n},\ldots,\tilde{X}_{r,n}).
\end{align*}
The kernel $h$ is called 
\emph{degenerate of level $k-1$} ($2 \leq k \leq r$) with regard to the measure $\P_n$ if 
\begin{align*}
\E h(x_1,\ldots,x_{k-1},\tilde{X}_{k,n},\ldots, \tilde{X}_{r,n}) = \theta_n
\end{align*}
for any $(x_1^\top,\ldots, x_{k-1}^\top)^\top\in\supp(\P_n^{k-1})$, the support of the product measure $\P_n^{k-1}$. 

When $h$ is degenerate of level $k - 1$, its Hoeffding decomposition takes the form
\begin{align*}
h(x_1,\ldots,x_r)  - \theta_n = \sum_{1\leq i_1<\ldots <i_k\leq r}h_{k,n}(x_{i_1},\ldots,x_{i_k}) + \ldots + h_{r,n}(x_1,\ldots,x_r),
\end{align*}
where $\{h_{p,n}\}_{p=k}^r$ are recursively defined as
\begin{equation*}
\begin{aligned}
&h_{1,n}(x) := g_{1,n}(x),\\
&h_{p,n}(x_1,\ldots,x_p) := g_{p,n}(x_1,\ldots,x_p) -\sum_{k=1}^p h_{1,n}(x_k)- \ldots - \sum_{1\leq k_1<\ldots <k_{p-1} \leq p}h_{p-1,n}(x_{k_1},\ldots,x_{k_{p-1}}),
\end{aligned}
\end{equation*}
for $p=2,\cdots,r$, with $\{g_{p,n}\}_{p=1}^r$ defined as $g_{r,n} := h - \theta_n$, and 
\begin{align*}
g_{p,n}(x_1,\ldots,x_p) := \E h(x_1,\ldots,x_p,\widetilde{X}_{p+1,n},\ldots,\widetilde{X}_{r,n}) - \theta_n
\end{align*}
for $1\leq p\leq r-1$. For each $1\leq p\leq r$, we denote the V-statistic generated by $h_{p,n}$ by
\begin{align*}
V_n(h_{p,n}) := n^{-p}\sum_{i_1,\ldots,i_p = 1}^n h_{p,n}(X_{i_1,n},\ldots,X_{i_p,n}).
\end{align*}


\begin{thm}[a slight modification to Theorem 1 in \cite{shen2019tail}]
Suppose $\{X_{i,n}\}_{i=1}^n$ is part of a stationary sequence $\{X_{t,n}\}_{t\in\Z}$ that 
is geometrically $\alpha$-mixing with coefficient
\begin{align*}
\alpha(\{X_{t,n}\}_{t\in\Z};m) \leq c_n\exp(-C_n m)~~~\text{for all }m\geq 1,
\end{align*}
where $c_n,C_n$ are two positive constants.  Suppose $h\in L_1(\R^{rd})$ is fixed, symmetric, continuous, and its Fourier transform $\widehat{h}(u):= \int h(x)e^{-2\pi iu^\top x}dx$ satisfies
\begin{align*}
\int_{\R^{rd}}\Big|\widehat{h}(u)\Big|\|u\|^qdu < \infty
\end{align*}
for some $q \geq 1$, where $\|\cdot\|$ represents the Euclidean norm. Then, there exists a positive constant $C_n'=C(r,c_n,C_n)$ such that for each $1\leq p \leq r$, and any $n\geq 2,$ $x > 0$, 
\begin{align*}
\P\Big(|V_n(h_{p,n})|\geq x\Big) \leq 6\exp\Big\{-\frac{C'_nnx^{2/p}}{A_{p,n}^{1/p} + x^{1/p}M_{p,n}^{1/p}}\Big\}
\end{align*}
with 
\begin{align*}
A_{p,n} = 2^{2r}\norm{\widehat{h}}_{L_1}^2\Big\{\frac{64c_n^{1/3}}{1-\exp(-C_n/3)}+\frac{(\log n)^4}{n}\Big\}^p~~{\rm and}~~M_{p,n} = 2^r\norm{\widehat{h}}_{L_1}(\log n)^{2p}.
\end{align*}
\end{thm}
As Theorem \ref{thm:212}, the dependence of $C'_n$ on $r,c_n,C_n$ in the above theorem could be explicitly calculated. We also note that, though $h(\cdot)$ itself is assumed to be fixed, the ``degenerate" kernels $h_{p,n}$ could depend on $n$ through the measure $\P_n$ in the triangular array setting, and hence the subscript $n$ is kept.

\section{A cautionary example}

Section \ref{sec:inequ} exemplifies the use of dependence measures to construct desired moment/probability inequalities for quantifying the statistical properties of procedures in a high-dimensional time series model. The problem then reduces to characterizing these dependence measures in a triangular array setting as highlighted at the beginning of Section \ref{sec:inequ}. As is apparent from reading their definitions, those dependence measures introduced in  \cite{doukhan1999new}, \cite{dedecker2004coupling}, and \cite{wu2005nonlinear} can be explicitly calculated. Therefore, the verification of those dependence measures, as were made in \citet[Section 3]{dedecker2007book}, \cite{wu2005nonlinear}, \cite{han2018moment}, and many other places, are obviously still valid under the high-dimensional triangular array framework. 

The verifications for the mixing conditions introduced at the beginning of Section \ref{sec:intro}, on the other hand, should be checked with caution under this new framework. In the following we will use the example of $\beta$-mixing to showcase this new challenge of high dimensionality in establishing mixing-type dependence for time series data. 

In literature, for a time series model that is fixed (i.e., not changing when more data points are observed), there have been a variety of results to establish bounds for $\beta$-mixing coefficients. See, for example, \cite{liebscher2005towards} for a review and \cite{chan2001chaos} for $\beta$-mixing of Markov processes. Let's focus on a particular example. Consider the following simple $d$-dimensional stationary Gaussian VAR(1) model:
\begin{align}\label{eq:var1}
X_t=\kappa X_{t-1}+E_t=\sum_{j=0}^{\infty}\kappa^{j}E_{t-j},  {\rm ~for~all~}t\in \Z.
\end{align}
Here the autocorrelation coefficient $\kappa\in\mathbb{R}$ is assumed to be fixed and satisfy $0<\kappa<1$ for simplicity, and the innovation noises $\{E_t\in\mathbb{R}^d\}$ are i.i.d. Gaussian. Then it is immediate (cf. Proposition 2 in \cite{liebscher2005towards}) that $\{X_t\}$ is geometrically $\beta$-mixing satisfying
\begin{align}\label{eq:fixed}
\beta(\{X_t\};m)\leq C\gamma^m
\end{align}
for some fixed constants $C>0, \gamma<1$.  

However, in high dimensions such a derivation is problematic. Let's fix the framework first. Adopting the triangular array setting as described in the last section, we assume that the studied model could change as more observations are available to us. In other words, let's adopt a parallel model to Equation \eqref{eq:var1}: for any $n=1,2,3,\ldots$, write
\begin{align}\label{eq:change}
X_{t,n}=
\begin{pmatrix}
X_{t,1,n}\\
X_{t,2,n}\\
\vdots\\
X_{t,d_n,n}
\end{pmatrix}
=\begin{pmatrix}
\kappa_n & 0& \ldots & 0\\
0 & \kappa_n& \ldots & 0\\
\vdots & \vdots & \ddots & 0\\
0 & 0 & 0 & \kappa_n
\end{pmatrix} \underbrace{\begin{pmatrix}
X_{t-1,1,n}\\
X_{t-1,2,n}\\
\vdots\\
X_{t-1,d_n,n}
\end{pmatrix}}_{X_{t-1,n}\in\reals^{d_n}}+\underbrace{\begin{pmatrix}
E_{t,1,n}\\
E_{t,2,n}\\
\vdots\\
E_{t,d_n,n}
\end{pmatrix}}_{E_{t,n}\in\reals^{d_n}}, {\rm ~for~all~}t\in \Z.
\end{align}
Here for any $n$, the observed data $\{X_{1,n}, X_{2,n}, \ldots,X_{n,n}\}$ are assumed to be generated from a process $\sum_{j=0}^{\infty}\kappa_n^{j}E_{t-j,n}$, where first of all the dimension of the time series $d_n$ has been allowed to change with the sample size $n$. Moreover, as an implicit consequence of the above high-dimensional triangular array framework, all the parameters in  Model \eqref{eq:change}, including $\kappa_n\in \mathbb{R}$, ${\rm Cov}(X_{t,n})\in \mathbb{R}^{d_n\times d_n}$, and ${\rm Cov}(E_{t,n})\in \mathbb{R}^{d_n\times d_n}$, are now allowed to change as the sample size $n$ is increasing.

Once such a framework is fixed, it becomes clear that the analysis of various dependence conditions has to be nonasymptotic, i.e., we now have to provide an analysis of the $\beta$-mixing coefficient that takes the change of $d_n, \kappa_n$, and all the other model parameters into account. 
With these concepts in mind, we first state a somehow comforting result that certain desirable properties could still be established for $\alpha$-mixing (in contrast to the $\beta$-mixing) coefficient under the triangular array setting.
\begin{thm}\label{thm:alpha}
Consider the following simple stationary Gaussian vector autoregressive model that generalizes \eqref{eq:change} by relaxing restrictions on the transition matrix:
\begin{align}\label{eq:VAR}
X_{t,n}=A_n X_{t,n-1}+E_{t,n},~~t\in\Z.
\end{align}
We then have
\begin{align*}
\alpha(\{X_{t,n}\}_{t\in \Z};m)\leq \Big\{\frac{\lambda_{\max}(\Sigma_n)}{\lambda_{\min}(\Sigma_n)}\Big\}^{1/2}\norm{A_n}^m,
\end{align*}
where $\Sigma_n:={\rm Cov}(X_{0,n})$, $\lambda_{\min}(\cdot)$ stands for the smallest eigenvalue of the input, and $\norm{\cdot}$ is the matrix spectral norm.
\end{thm}
\begin{proof}
For notation simplicity, let's remove $n$ from the subscript. Since VAR(1) is a stationary Markov chain, by \cite{bradley2005basic}, we have the $\rho$-mixing coefficient
	\begin{align*}
	\rho\{\sigma(X_{-\infty}^{0}), \sigma(X_{m}^{\infty})\} = \rho\{\sigma(X_0), \sigma(X_m)\}.
	\end{align*}
	By Theorem 1 from \cite{kolmogorov1960strong}, if $U_1,\ U_2,\ \dots,\ U_m,\ V_1,\ V_2,\ \dots,\ V_\ell$ are jointly normal random variables, then there exist real numbers $a_1,\ a_2,\ \dots,\ a_m,\ b_1,\ b_2,\ \dots,\ b_\ell$ such that
	\begin{align*}
	\rho\{\sigma(U_k, 1 \leq k \leq m), \sigma(V_k, 1 \leq k \leq \ell)\} = \mbox{Corr}\Big(\sum_{k = 1}^m a_kU_k, \sum_{k = 1}^\ell b_k V_k\Big).
	\end{align*} 
	Since $(X_0, X_m)$ is multivariate normal, there exist real numbers $a = (a_1, a_2, \dots, a_p)^\top, b=(b_1, b_2, \dots, b_p)^\top$ such that
	\begin{align*}
	&\rho\{\sigma(X_0), \sigma(X_m)\}= \mbox{Corr} (a^\top X_0, b^\top X_m)= \frac{a^\top \Sigma (A^m)^\top b}{\sqrt{a^\top\Sigma a b^\top \Sigma b}}
	\leq \sqrt{\lVert \Sigma^{\frac{1}{2}}(A^m)^\top  \Sigma^{-1}(A^m) \Sigma^{\frac{1}{2}}\rVert},
	\end{align*}
	where the last inequality is followed by Cauchy-Schwarz. Hence we have
	\begin{align*}
	&\rho\{\sigma(X_{-\infty}^0), \sigma(X_{m}^{\infty})\}\leq \Big\{\frac{\lambda_{\max}(\Sigma)}{\lambda_{\min}(\Sigma)}\Big\}^{\frac{1}{2}}\lVert A\rVert^m.
	\end{align*}
	Now noticing
	\begin{align*}
	\alpha\{\sigma(X_{-\infty}^0), \sigma(X_{m}^{\infty})\} \leq \rho\{\sigma(X_{-\infty}^0), \sigma(X_{m}^{\infty})\}
	\end{align*}
	finishes the proof.
\end{proof}

Applying Theorem \ref{thm:alpha} to Model \eqref{eq:change}, it is clear that the $\alpha$-mixing coefficient for Model \eqref{eq:change} is bounded by $\kappa_n^m$, which will be exponentially tending to 1 if $\kappa_n<1$ is fixed, regardless of how large the dimension $d_n$ is.  We then state a possibly striking result, that, even if restricting to the Model \eqref{eq:change} and fixing $\kappa_n$, the $\beta$-mixing coefficient of the time series $\{X_{t,n}\}_{t\in \Z}$ could still be tending to 1 if $d_n$ is sufficiently larger than the time gap. Thusly, $\beta$-mixing coefficient is dimension-dependent.

\begin{thm}\label{thm:key}
Consider the model \eqref{eq:change} under the triangular array setting. If we further assume that $E_{0,n}=(E_{0,1,n},\ldots,E_{0,d_n,n})^\top$ have i.i.d. components, then for any positive integers $n$ and $m$, we have 
\[
\beta(\{X_{t,n}\}_{t\in\Z}; m)\geq 1-2\exp\Big(-\frac{d_n\kappa_n^{2m}}{18\pi^2} \Big).
\]
In particular, if (1) $d_n=d$ is not changing with $n$ but $\lim_{n\to\infty}\kappa_n^{2n}>\frac{18\pi^2\log 2}{d}$, or (2) $\kappa_n=\kappa$ is not changing with $n$ but $\lim_{n\to\infty} d_n\kappa^{2n}>18\pi^2\log 2$, then $\liminf_{n\to\infty}\beta(\{X_{t,n}\}_{t\in\Z}; n) > 0$.
\end{thm}

Theorem \ref{thm:key} is concerning a particularly simple model that is merely aggregating $d_n$ i.i.d. AR(1) Gaussian sequences once we have $n$ data points. It is very unlikely that any assumption in a general theorem for quantifying the behavior of a high-dimensional time series could exclude such a simple case. However,  it has been apparent from this result that, once the triangular array framework is adopted, many simple and elegant properties like Equation \eqref{eq:fixed} could no longer be trusted because otherwise, the case that $\lim_{n\to\infty}\beta(\{X_{t,n}\}_{t\in\Z}; n)\ne 0$ shall never happen. The reason is, once the model $\{X_{t,n}\}$ is allowed to change with $n$, the values of $C$ and $\gamma$ in \eqref{eq:fixed} will depend on the sample size $n$. Any solid analysis of the $\beta$-mixing coefficient hence has to be fully nonasymptotic. This, however, violates the spirit beneath the definition of various mixing coefficients, and to the authors' knowledge, cannot be trivially handled (except for the $\alpha$- and $\rho$-mixing coefficients under a Gaussian process, as showcased above in Theorem \ref{thm:alpha}).

\begin{proof}[Proof of Theorem \ref{thm:key}]
The proof is nonasymptotic and relies on several known results in the mixing literature. For presentation clearness, we omit the subscript $n$ in the following when no confusion is made.

In the first step, we need to establish a lower bound for the marginal $\alpha$-mixing coefficient $\alpha(\sigma(X_{0,1}), \sigma(X_{m,1}))$ with the understanding that
\[
X_{t}=(X_{t,1},\ldots,X_{t,d})^\top.
\]
For any bivariate Gaussian random vector $(Z_1,Z_2)^\top\in\mathbb{R}^2$, the following two facts are known. 
\begin{itemize}
\item[(1)] One has
\[
\alpha(\sigma(Z_1),\sigma(Z_2))\leq \rho(\sigma(Z_1),\sigma(Z_2))\leq 2\pi \alpha(\sigma(Z_1),\sigma(Z_2)).
\]
See, for example, Equation (1.9) in \citet[Chapter 4]{ibragimov2012gaussian} or Theorem 2 in \cite{kolmogorov1960strong}.
\item[(2)] Theorem 1 in \cite{kolmogorov1960strong} gives
\[
\rho(\sigma(Z_1),\sigma(Z_2))=|{\rm Corr}(Z_1,Z_2)|.
\]
\end{itemize}
The above two results then yield
\begin{align*}
\alpha(\sigma(X_{0,1}), \sigma(X_{m,1}))&\geq \frac{1}{2\pi}|{\rm Corr}(X_{0,1},X_{0,m})|=\frac{\kappa^m}{2\pi}.
\end{align*}

In the second step, we are going to establish a lower bound on $\beta(\sigma(X_{0}), \sigma(X_{m}))$ based on the derived lower bound for the marginal $\alpha$-mixing coefficient. For any $j\in [d]$, since
\[
\alpha(\sigma(X_{0,j}), \sigma(X_{m,j}))\geq \frac{\kappa^m}{2\pi},
\]
by definition, there must exist sets $G\in \sigma(X_{0,j})$ and $H \in \sigma(X_{m,j})$ such that
\[
\Big|\P\Big(X_{0,j}\in G, X_{m,j}\in H\Big) - \P\Big(X_{0,j}\in G\Big)\P\Big(X_{m,j}\in H\Big)\Big| \geq \frac{\kappa^m}{3\pi}=:\eta.
\]
Without loss of generality, we may assume that 
\[
\theta:=\P(X_{0,j}\in G_n, X_{m,j}\in H)>\P(X_{0,j}\in G)\P(X_{m,j}\in H)=:\xi.
\]
For $j\in[d]$, let's define $V_j, W_j$ as
\[
V_j=\ind(X_{0,j}\in G)~~~{\rm and}~~~W_j=\ind(X_{m,j}\in H),
\]
where $\ind(\cdot)$ represents the indicator function. Then we have 
\begin{itemize}
\item[(1)] For each $j\in [d]$, $\E V_j=\P(X_{0,j}\in G)$, $\E W_j=\P(X_{m,j}\in H)$, $\E (V_jW_j)=\theta$.
\item[(2)] By i.i.d.-ness of $E_{0,1},\ldots,E_{0,d}$, $\{(V_j, W_j),j\in[d]\}$ is an i.i.d. sequence.
\end{itemize}

Let's now consider the following event
\begin{align}\label{eq:event1}
\Big\{\frac{1}{d}\sum_{j=1}^{d}V_jW_j \geq \theta- \frac{\eta}{2} \Big\}.
\end{align}
By Hoeffding's inequality for i.i.d. data \citep{hoeffding1963probability}, we have 
\[
\P\Big(\frac{1}{d}\sum_{j=1}^{d}V_jW_j \geq \theta- \frac{\eta}{2}\Big) \geq 1-\exp(-d\eta^2/2).
\]

On the other hand, consider a comparable event to \eqref{eq:event1} under the product measure:
\[
\Big\{\frac{1}{d}\sum_{j=1}^{d}V_j\tilde W_j \geq \theta- \frac{\eta}{2} \Big\},
\]
where $\{\tilde W_j\}_{j\in[d]}$ is a copy of $\{W_j\}_{j\in[d]}$ and is independent of $\{V_j\}_{j\in[d]}$. Again, by Hoeffding's inequality, we have 
\[
\P\Big(\frac{1}{d}\sum_{j=1}^{d}V_j\tilde W_j \geq \theta- \frac{\eta}{2} \Big)\leq \P\Big(\frac{1}{d}\sum_{j=1}^{d}V_j\tilde W_j \geq \xi+ \frac{\eta}{2} \Big)\leq \exp(-d\eta^2/2)
\]
as $\theta-\xi\geq \eta$.

By definition of $\beta$-mixing coefficient, we then have
\begin{align*}
\beta(\sigma(X_{0}), \sigma(X_{m}))&\geq \Big|\P\Big(\frac{1}{d}\sum_{j=1}^{d}V_jW_j \geq \theta- \frac{\eta}{2}\Big)- \P\Big(\frac{1}{d}\sum_{j=1}^{d}V_j\tilde W_j \geq \theta- \frac{\eta}{4} \Big)\Big|\\
&\geq 1-2\exp\Big(-\frac{d\kappa^{2m}}{18\pi^2} \Big).
\end{align*}

Lastly, by noticing that the model studied is naturally a Markov chain and by using Theorem 7.3 in \cite{bradley2007}, one obtains
\[
\beta(\{X_{t}\};m)=\beta(\sigma(X_{0}), \sigma(X_{m})),
\]
which finishes the proof.
\end{proof}

\bibliographystyle{apalike}
\bibliography{mix}


\end{document}